\documentclass[letter]{amsart}

\usepackage[latin1]{inputenc}
\usepackage{subfigure}
\usepackage{enumitem}
\usepackage{graphicx}

\usepackage{amssymb,stmaryrd,comment,enumitem, amsmath, amsthm}

\newtheorem{theorem}{Theorem}
\newtheorem{lemma}[theorem]{Lemma}
\newtheorem{proposition}[theorem]{Proposition}
\newtheorem{corollary}[theorem]{Corollary}

\newtheorem{remark}{Remark}

\author{Sean Carrell}\thanks{S.C.: Department of Combinatorics \& Optimization, University of Waterloo, Waterloo, Canada. Email: {\tt s.r.carrell@gmail.com}}
\author{Guillaume Chapuy}\thanks{G.C.: CNRS and LIAFA, Universit\'e Paris Diderot -- Paris 7, Paris, France. Partial support from
French \emph{Agence
Nationale de la Recherche}, grant number ANR 12-JS02-001-01 ''Cartaplus'', and from
\emph{Ville de Paris}, grant ''\'Emergences Paris 2013, Combinatoire \`a
Paris''. Email:
{\tt guillaume.chapuy@liafa.univ-paris-diderot.fr}.}

\begin{document}
\title[Simple recurrence formulas to count maps on orientable surfaces]{Simple recurrence formulas to count maps on orientable surfaces.
}

\maketitle

\begin{abstract}
We establish a simple recurrence formula for the number $Q_g^n$
of rooted orientable maps counted by edges and genus.
We also give a weighted variant for the generating polynomial $Q_g^n(x)$ where $x$ is a parameter taking the number of faces of the map into account, or equivalently a simple recurrence formula for the refined numbers $M_g^{i,j}$ that count maps by genus, vertices, and faces.
These formulas give
by far the fastest known way of computing these numbers, or the fixed-genus
generating functions, especially for large~$g$. 
In the very particular case of one-face maps, we recover the Harer-Zagier recurrence formula.

Our main formula is a 
consequence of the KP equation for the generating function of bipartite maps,
coupled with a Tutte equation, and it was apparently unnoticed before.
It is similar in look to the one discovered by Goulden and Jackson for
triangulations, and indeed our method to go from the KP equation to the recurrence formula can be seen as a combinatorial simplification of Goulden and Jackson's approach (together with one additional combinatorial trick).
All these  formulas have a very combinatorial flavour, but finding a
bijective interpretation is currently unsolved.
\end{abstract}

\section{Introduction and main results}

A \emph{map} is a connected graph embedded in a compact connected orientable
surface in such a way that the regions delimited by the graph, called
\emph{faces}, are homeomorphic to open discs. Loops and multiple edges are
allowed. A \emph{rooted} map is a map in which an angular sector incident to a
vertex is distinguished, 
and the latter is called the \emph{root vertex}. The \emph{root edge} is the edge encountered
when traversing the distinguished angular sector clockwise around the root vertex. 
Rooted maps are considered up to oriented
homeomorphisms preserving the root sector.  

A map is \emph{bipartite} if its vertices can be coloured with two colors, say
black and white, in such a way that each edge links 
a white and a black vertex.
Unless otherwise mentioned, bipartite maps will be endowed with their
\emph{canonical bicolouration} in which the root vertex is coloured white.
The \emph{degree} of a face in a map is equal to the number of edge sides along
its boundary, counted with multiplicity. 
Note that in a bipartite map every face has even degree, since colours
alternate along its boundary.

A \emph{quadrangulation} is a map in which every face has degree~$4$. 
 There is a classical
bijection,
that goes back to Tutte~\cite{Tutte:census}, between bipartite
quadrangulations with $n$ faces and genus $g$, and rooted maps with $n$ edges
and genus $g$. It is illustrated on Figure~\ref{fig:Tutte}.
This bijection transports the number of faces of the map to the number of white
vertices of the quadrangulation (in the canonical bicolouration).
\medskip

For $g,n\geq 0$, we let $Q_g^n$ be the number of rooted bipartite quadrangulations of genus $g$
with $n$ faces. Equivalently, by Tutte's construction, $Q_g^n$ is the number of 
rooted maps of genus $g$ with $n$ edges. 
By convention we admit a single map with no edges and which has genus zero, one face, and one vertex.
Our first result is the following recurrence formula:
\begin{theorem}\label{thm:main}
The number $Q_g^n$ of rooted maps of genus $g$ with $n$ edges 
(which is also the number of rooted bipartite quadrangulations of genus $g$ with $n$ faces) satisfies the
following recurrence relation:
\begin{multline*}
\!\!
\!\!
\!\!
\frac{n+1}{6} Q_g^n = \frac{4n-2}{3} 
  Q_g^{ n-1 } + \frac{ ( 2n-3 )  ( 2n-2 )  ( 2n-1 ) }{12}Q_{g-1}^{ n-2 } 
+
\frac{1}{2}\sum _{k+\ell=n\atop k,\ell\geq 1}  \sum _{ i+j=g\atop i,j\geq0}
 ( 2k-1 )  ( 2\ell-1 )   Q_{i}^{k-1}  Q_{j}^{ \ell-1}, 
\end{multline*}
for $n\geq 1$, with the initial conditions $Q_g^0={\bf 1}_{\{g=0\}}$, and $Q_g^n=0$ if $g<0$ or $n<0$.
\end{theorem}

We actually prove a more general result, where in addition to edges and genus, we also control the number of faces of the map. Let $x$ be a formal variable, and let $Q_g^n(x)$ be the generating polynomial of maps of genus $g$ with $n$ edges, where the exponent of $x$ records the number of faces of the map:
\begin{eqnarray}\label{eq:qngx}
Q_g^n(x) := \sum_{\mathfrak{m}} x^{\#\mbox{faces of }\mathfrak{m}},
\end{eqnarray}
where the sum is taken over rooted maps of genus $g$ with $n$ edges. We then have the following generalization of Theorem~\ref{thm:main}:
\begin{theorem}\label{thm:mainx}
The generating polynomial $Q_g^n(x)$ of rooted maps of genus $g$ with $n$ edges and a weight $x$ per face
(which is also the generating polynomial of rooted bipartite quadrangulations of
genus $g$ with $n$ faces with a weight $x$ per white vertex) satisfies the
following recurrence relation:
\begin{multline*}
\!\!
\!\!
\!\!
\frac{n+1}{6} Q_g^n (x) = \frac{(1+x)(2n-1)}{3} 
  Q_g^{ n-1 }(x) + \frac{ ( 2n-3 )  ( 2n-2 )  ( 2n-1 ) }{12}Q_{g-1}^{ n-2 }(x) \\ 
+
\frac{1}{2}\sum _{k+\ell=n\atop k,\ell\geq 1}  \sum _{ i+j=g\atop i,j\geq0}
 ( 2k-1 )  ( 2\ell-1 )   Q_{i}^{k-1}(x)  Q_{j}^{ \ell-1} (x) , 
\end{multline*}
for $n\geq 1$, with the initial conditions $Q_g^0(x)=x\cdot {\bf 1}_{\{g=0\}}$, and $Q_g^n=0$ if $g<0$ or $n<0$.
\end{theorem}
Of course, Theorem~\ref{thm:main} is a straightforward corollary of Theorem~\ref{thm:mainx} (it just corresponds to the case $x=1$).
By extracting the coefficient of $x^f$ in Theorem~\ref{thm:mainx}, for $f\geq
1$, we obtain yet another corollary that enables one to count maps by edges,
vertices, and genus:
\begin{corollary}\label{cor:three}
The number $Q_g^{n,f}$ of  rooted maps of genus $g$ with $n$ edges and $f$ faces
(which is also the number of rooted bipartite quadrangulations of genus $g$ with $n$ faces and $f$ white vertices) satisfies the
following recurrence relation:
\begin{multline*}
\!\!
\!\!
\!\!
\frac{n+1}{6} Q_g^{n,f}  = 
\frac{(2n-1)}{3}   Q_g^{ n-1,f }
+
\frac{(2n-1)}{3}   Q_g^{ n-1,f-1 }
 + \frac{ ( 2n-3 )  ( 2n-2 )  ( 2n-1 ) }{12}Q_{g-1}^{ n-2,f } \\ 
+
\frac{1}{2}\sum _{k+\ell=n\atop k,\ell\geq 1} \sum_{u+v=f, \atop u,v \geq 1}
  \sum _{ i+j=g\atop i,j\geq0}
( 2k-1 )  ( 2\ell-1 )   Q_{i}^{k-1,u}  Q_{j}^{ \ell-1,v} , 
\end{multline*}
for $n,f\geq 1$, with the initial conditions $Q_g^{0,f}={\bf 1}_{\{(g,f)=(0,1)\}}$ and $Q_g^{n,f}=0$ whenever $f,g$, or $n$ is negative.
\end{corollary}
\noindent Corollary~\ref{cor:three} has interesting specializations when the number of faces $f$ is small. In particular, when $f=1$, the equation becomes linear, and one recovers the celebrated Harer-Zagier formula (\cite{HZ}, see \cite{CFF} for a bijective proof):
\begin{corollary}[Harer-Zagier recurrence formula, \cite{HZ}]\label{cor:HZ}
The number $\epsilon_g(n)=Q_g^{n,1}$ of  rooted maps of genus $g$ with $n$ edges and one face
satisfies the
following recurrence relation:
\begin{align*}
\frac{n+1}{6} \epsilon_g(n) = 
\frac{(2n-1)}{3}   \epsilon_g(n-1)
 + \frac{ ( 2n-3 )  ( 2n-2 )  ( 2n-1 ) }{12}\epsilon_{g-1}(n-2) , 
\end{align*}
with the initial conditions $\epsilon_g(0)={\bf 1}_{\{g=0\}}$ and $\epsilon_g(n)=0$ if $n<0$ or $g<0$.
\end{corollary}
\noindent We conclude this list of corollaries with yet another formulation of Corollary~\ref{cor:three} that takes a nice symmetric form and emphasizes the duality between vertices and faces inherent to maps.
Let $M_g^{i,j}$ be the number of rooted maps of genus $g$ with $i$ vertices and $j$ faces. Euler's relation ensures that such a map has $n$ edges where:
$$
i+j = n+2-2g,
$$
which shows that $M_{g}^{i,j} = Q_{g}^{i+j+2g-2,j}$. Corollary~\ref{cor:three} thus takes the following form:
\begin{theorem}\label{thm:three2}
The number $M_g^{i,j}$ of  rooted maps of genus $g$ with $i$ vertices and $j$ faces
(which is also the number of rooted bipartite quadrangulations of genus $g$ with $i$ black vertices and $j$ white vertices) satisfies the
following recurrence relation:
\begin{multline*}
\hspace{5mm}
\frac{n+1}{6} M_g^{i,j}  = 
\frac{(2n-1)}{3}   \Big( M_g^{ i-1,j }
+
 M_g^{ i,j-1 }
 + \frac{ ( 2n-3 )  ( 2n-2 ) }{4}M_{g-1}^{ i,j }  
 \Big)\\
+
\frac{1}{2}\sum _{i_1+i_2=i\atop i_1,i_2\geq 1}  
\sum_{j_1+j_2=j, \atop j_1,j_2 \geq 1}\sum _{ g_1+g_2=g\atop g_1,g_2\geq0}
 ( 2n_1-1 )  ( 2n_2-1 )   M_{g_1}^{i_1,j_1}  M_{g_2}^{ i_2,j_2} , \hspace{5mm} 
\end{multline*}
for $i,j\geq 1$, 
with the initial conditions that $M_g^{i,j}=0$ if $i+j+2g<2$, that if $i+j+2g=2$ then 
$M_g^{i,j}={\bf 1}_{\{(i,j)=(1,1)\}}$, and where we use the notation $n=i+j+2g-2$, $n_1=i_1+j_1+2g_1-1$, and  $n_2=i_2+j_2+2g_2-1$.
\end{theorem}

The rest of the paper is organized as follows. In Section~\ref{sec:mainproof},
we prove Theorem~\ref{thm:mainx} (and therefore all the theorems and corollaries stated
above). This result relies on both classical facts about the
KP equation for bipartite maps, and an elementary Lemma obtained by
combinatorial means~(Lemma~\ref{lemma:only}).
In Section~\ref{sec:gencor}, we give corollaries of our results
in terms of generating functions. In particular, we
obtain a very efficient recurrence formula that can be used to compute the generating function
of maps of fixed genus inductively (Theorem
\ref{thm:recgf}).
 Finally, in Section~\ref{sec:comments}, we
comment on the differences between what we do here and other known approaches to the
problem: in brief, our method is much more powerful for the particular problem
treated here, but we still don't know whether it can be applied successfully to
cases other than bipartite quadrangulations.

\begin{figure}[h]
\begin{center}
\includegraphics[width=12cm]{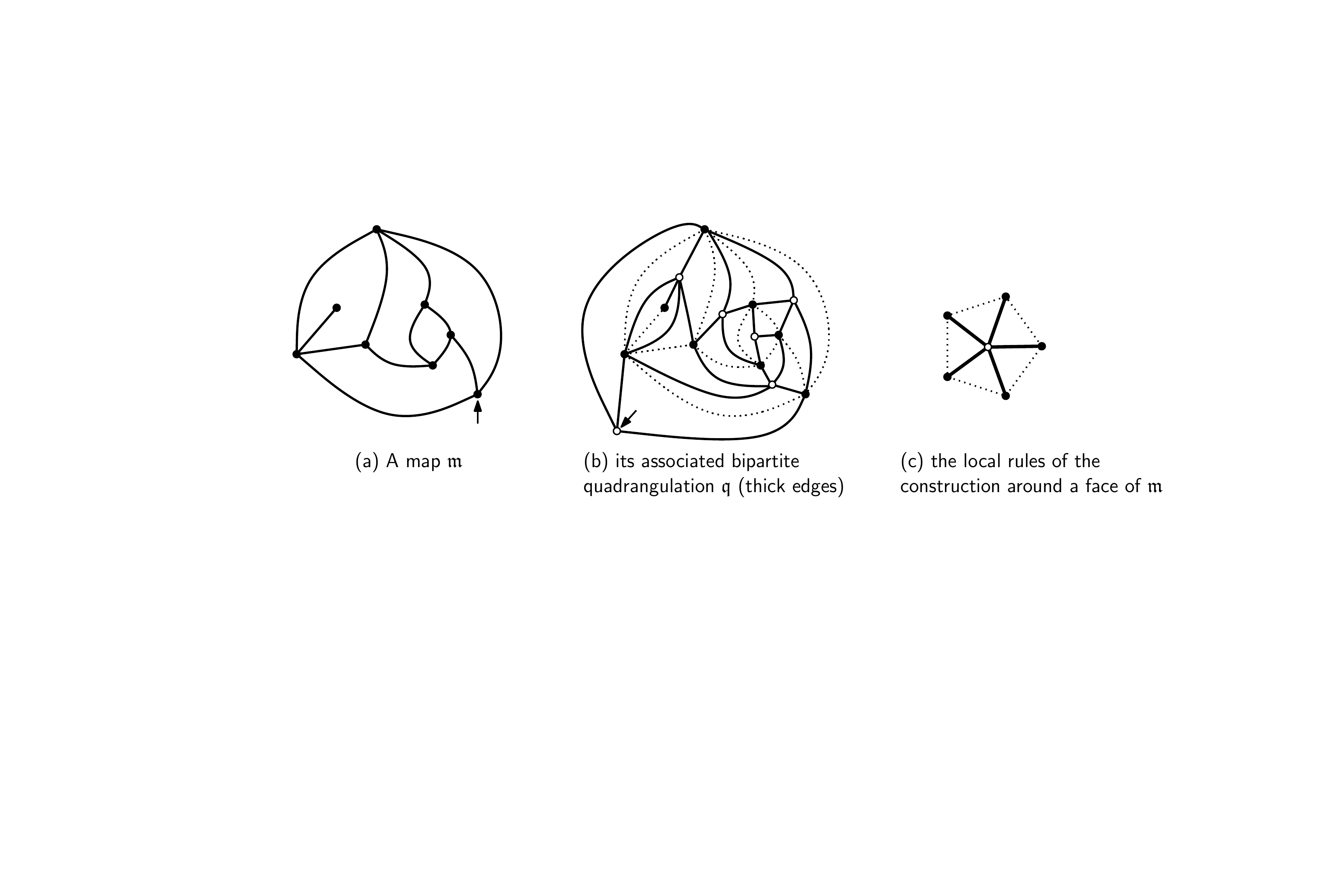}
\caption{{\bf Tutte's bijection.}
Given a (not necessarily bipartite) map $\mathfrak{m}$ of genus $g$ with $n$
edges, add a
new (white)
vertex inside each face of $\mathfrak{m}$, and link it by a new edge to each of the corners
incident to the face. The bipartite quadrangulation $\mathfrak{q}$
is obtained by
erasing all the original edges of $\mathfrak{m}$, \emph{i.e.} by keeping only
the new (white) vertices, the old (black) vertices, and the newly created
edges. The root edge of $\mathfrak{q}$ is the one created from the root corner
of $\mathfrak{m}$ (which is enough to root $\mathfrak{q}$ if we demand that its
root vertex is white).
(a) and (b) display an example of the construction for a map
of genus $0$ (embedded on the sphere). Root corners are indicated by arrows.
  }\label{fig:Tutte}
\end{center}
\end{figure}

\medskip \noindent{\bf Acknowledgements.} The first version of this paper dealt only
with the numbers $Q_g^n(1)$ without keeping track of the number of faces ({\it
i.e.} it contained Theorem~\ref{thm:main} but neither Theorem~\ref{thm:mainx}
nor its other corollaries). We are very grateful to \'Eric Fusy for asking to us whether 
we could control the number of faces as well, and to the organizers of the meeting \emph{Enumerative Combinatorics} 
in Oberwolfach (March 2014) where this question was asked.

\section{Proof of the main formula}
\label{sec:mainproof}

\subsection{Bipartite maps and KP equation}

The first element of our proof is the fact that the generating function for
bipartite maps is a solution to the KP equation (Proposition~\ref{prop:KP}
below). 
In the rest of the paper, the \emph{weight} of a map is one over its number of
edges, and a generating function of some family of maps is \emph{weighted} if
 each map 
is counted with its weight in this generating function. We let $z$, $w$, $x$ and
${\bf p}=p_1,p_2,\dots$ be infinitely many indeterminates. We extend the
variables in ${\bf p}$ multiplicatively to partitions, {\it i.e.} we denote
$p_\alpha:=\prod_i p_{\alpha_i}$ if $\alpha$ is a partition.
The keystone of this paper is the following result%
\footnote{the literature on
the KP hierarchy has been built over the years, with many references written by 
mathematical physicists and published in the physics literature. This is
especially true for the link with map enumeration, often arising in formal
expansions of matrix integrals. Thus it is not always
easy for the mathematician to know who to attribute the results in this field.
The reader may consult~\cite[Chapter 5]{LandoZvonkine} for historical references related to matrix
integrals in the physics literature, and \cite{GJ2008,
CarrellGoulden} for self-contained proofs written in the 
language of algebraic combinatorics.
As for Proposition~\ref{prop:KP}, it is essentially a consequence of the
classical fact
that map generating functions can be written in terms of Schur
functions (see e.g. ~\cite{JV1990a}), together with a result of Orlov and
Shcherbin~\cite{Orlov} that implies that certain infinite linear combinations of
Schur functions satisfy the KP hierarchy. To be self-contained here, we have
chosen to give the most easily checkable reference, and we prove Proposition~\ref{prop:KP} 
by giving all the details necessary to make the link with an equivalent statement in~\cite{GJ2008}.
}.
\begin{proposition}[\cite{GJ2008}, see also \cite{Orlov}
]\label{prop:KP}
For $n, v, k \geq 1$, and $\alpha \vdash n$ a partition of $n$, let $H_\alpha(n, v, k)$ be
the number of rooted bipartite maps with $n$ edges and $v$ vertices, $k$ of which are
white, where the half face degrees are given by the parts of $\alpha$.
Let $H=H(z,w,x;{\bf p})$ be
the weighted generating function of bipartite maps, with $z$ marking edges, $w$
marking vertices, $x$ marking white vertices and the $p_i$ marking the number of faces of degree $2i$ for
$i\geq 1$:
$$
H(z,w,x;{\bf p}):= 1+ \sum_{{n\geq 1\atop v\geq 1 }\atop k \geq 1} \frac{w^v z^n x^k}{n} \sum_{\alpha \vdash n}
H_\alpha(n,v,k) p_\alpha.
$$
Then $H$ is a solution of the KP equation:
\begin{eqnarray}\label{eq:kp}
-H_{3,1}+H_{2,2}+\frac{1}{12}H_{1^4} +\frac{1}{2}(H_{1,1})^2 =0,
\end{eqnarray}
where indices indicate partial derivatives with respect to the variables $p_i$,
for example $H_{3,1}:=\frac{\partial^2}{\partial p_3 \partial p_1} H$.
\end{proposition}
Actually, the generating
function $H$ is a solution of an infinite system of partial differential equations,
known as the KP Hierarchy (see, e.g., \cite{MJD,GJ2008, CarrellGoulden}), but we will need only the simplest one of these
equations here, namely~\eqref{eq:kp}.

\begin{proof}
First recall that a bipartite map $\mathfrak{m}$ with $n$ edges labelled
from $1$ to $n$ can be encoded by a triple of permutations
$(\sigma_\circ,\sigma_\bullet, \phi) \in (S_n)^3$ such that
$\sigma_\circ\sigma_\bullet = \phi$. In this correspondence, the cycles of
the permutation $\sigma_\circ$ (resp. $\sigma_\bullet$) encode the
counterclockwise ordering of the edges around the white (resp. black)
vertices of $\mathfrak{m}$, while the cycles of $\phi$ encode the clockwise
ordering of the white to black edge-sides around the faces of $\mathfrak{m}$.
This encoding gives a $1$ to $(n-1)!$ correspondence between rooted bipartite maps
with $n$ edges and triples of permutations as above that are \emph{transitive},
i.e. that generate a transitive subgroup of $S_n$. We refer to~\cite{CM}, or
Figure~\ref{fig:encoding} 
for more about this encoding (see also
\cite{LandoZvonkine, JV1990a}).

Now recall Theorem 3.1 in \cite{GJ2008}. Let
$b_{\alpha, \beta}^{(a_1, a_2, \cdots)}$ be the number of tuples of
permutations $(\sigma, \gamma, \pi_1, \pi_2, \cdots)$ on $\{1, \cdots, n\}$
such that
\begin{enumerate}[itemsep=-1pt, topsep=1pt]
  \item $\sigma$ has cycle type $\alpha$, $\gamma$ has cycle type $\beta$
        and $\pi_i$ has $n - a_i$ cycles for each $i \geq 1$;
  \item $\sigma \gamma \pi_1 \pi_2 \cdots = 1$ in $S_n$ where 1 is the identity;
  \item the subgroup generated by $\sigma, \gamma, \pi_1, \pi_2, \cdots$
        acts transitively on $\{1, \cdots, n\}$.
\end{enumerate}
Then the series
$$
  B=\sum_{\substack{|\alpha| = |\beta| = n \geq 1, \\ a_1, a_2, \cdots \geq 0}}
           \frac{1}{n!} b_{\alpha, \beta}^{(a_1, a_2, \cdots)}
               p_\alpha q_\beta u_1^{a_1} u_2^{a_2} \cdots
$$
is a solution to the KP hierarchy in the variables $p_1, p_2, \cdots$.
Here $q_1,q_2,\dots$ and $u_1,u_2,\dots$ are two infinite sets of auxiliary variables, and we use the notation $q_\beta=\prod_i q_{\beta_i}$. 

Now, using the encoding of maps as triples of permutations described above,
we see that
$(n-1)!H_\alpha(n,v,k)=b_{\alpha, 1^n}^{(n-k, n+k-v, 0, \cdots)}$, since the coefficient on the right hand side is the number of solutions to
the equation $\sigma \gamma \pi_1 \pi_2 = 1$ where the total number of cycles
in $\pi_1$ and $\pi_2$ is $v$, $\pi_1$ has $k$ cycles, $\sigma$ has cycle type $\alpha$ and where
$\gamma$ is the identity. Multiplying by $\sigma^{-1}$ then gives
$\pi_1 \pi_2 = \sigma^{-1}$ which matches the encoding of bipartite maps
given above. Thus, by setting $q_1 = w^2zx$, $q_i = 0$ for
$i \geq 2$, $u_1 = w^{-1} x^{-1}, u_2 = w^{-1}$ and $u_i = 0$ for $i \geq 3$ in $B$,
we get the series $H$ as required.

Note that we choose to attribute this result to \cite{GJ2008} since this provides a clear and checkable
mathematical reference. The result
was referred to before this reference in the mathematical physics literature, however, it is hard
to find references in which the result is properly stated or proved. We refer
to Chapter 5 of the book~\cite{LandoZvonkine} as an entry point for the interested reader.

\end{proof}
\begin{figure}[h]
\begin{center}
\includegraphics{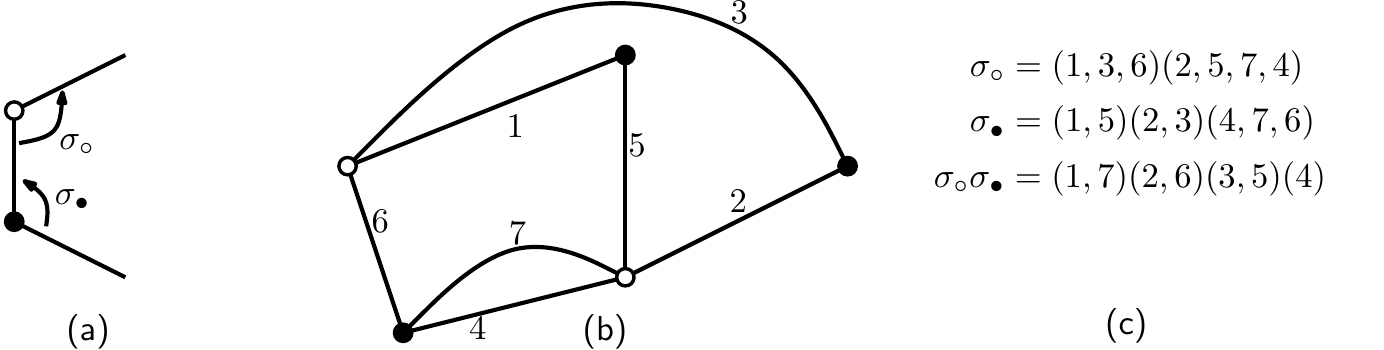}
\caption{(a) The rules defining the permutations $\sigma_\circ$ and $\sigma_\bullet$. (b) A bipartite map with $7$ edges arbitrarily labelled from $1$ to $7$. (c) The corresponding permutations $\sigma_\circ$ and $\sigma_\bullet$.}\label{fig:encoding}
\end{center}
\end{figure}

\subsection{Bipartite quadrangulations}

Our goal is to use Proposition~\ref{prop:KP} to get information on the
generating function of bipartite
quadrangulations. To this end, we let $\theta$ denote the operator 
that substitutes the variable~$p_2$ to~$1$
and all the variables~$p_i$ to $0$ for $i\neq 2$. 
When we apply $\theta$ to~\eqref{eq:kp} we get four terms:
\begin{eqnarray}\label{eq:thetakp}
-\theta H_{3,1}+\theta H_{2,2}+\frac{1}{12} \theta H_{1^4} + \frac{1}{2}(\theta H_{1,1})^2 =0.
\end{eqnarray}
Note that since all the derivatives appearing in $\eqref{eq:kp}$ are with respect
to $p_1, p_2$ or $p_3$, any monomial in $H$ that contains a variable $p_i$ for
some $i\neq \{1,2,3\}$ gives a zero contribution to \eqref{eq:thetakp}.
Therefore each of the four terms appearing in \eqref{eq:thetakp} can be
interpreted as the generating function of {\it some} 
family of bipartite maps having only faces of degree $2$,$4$, or $6$ (subject to
further restrictions). 
However, thanks to local operations on maps, we will be able to relate each term
to maps having only faces
of degree~$4$, as shown by the next lemma. 

If $A(z,w)$ is a formal power series in $z$ and $w$ with coefficients in $\mathbb{C}[x]$ we denote
by $[z^p w^q] A(z,w)$ the coefficient of the monomial $z^p w^q$ in $A(z,w)$. It is a polynomial in $x$.

\begin{lemma}\label{lemma:only}
Let $n, g \geq 1$. Then we have:
\begin{align} 
[z^{2n} w^{n+2-2g}] \theta H_{2,2} &=  \frac{n-1}{2} Q_g^n(x),  \label{eq:22}\\
[z^{2n} w^{n+1-2g}] \theta H_{1,1} &=   (2n-1) Q_g^{n-1}(x),   \label{eq:11}\\
[z^{2n} w^{n+2-2g}] \theta H_{1^4} &=   (2n-1)(2n-2)(2n-3)Q_{g-1}^{n-2}(x),
\label{eq:1111}\\
[z^{2n} w^{n+2-2g}] \theta H_{3,1} &=   
\frac{2n-1}{3} \Big( Q_g^n(x) - (1+x) Q_g^{n-1}(x)\Big).  \label{eq:31}
\end{align}
\end{lemma}
We now prove the lemma.
By definition, if $v\geq 1$ and $\lambda=(\lambda_1,\lambda_2,\dots,\lambda_\ell)$ 
is a partition of some integer, then 
$[z^{2n}w^v]\theta H_\lambda$
is $\frac{1}{2n}$ times the generating polynomial (the variable $x$ marking white vertices) of rooted bipartite maps with 
$2n$ edges, $v$ vertices, $\ell$ marked
(numbered) faces of degrees $2\lambda_1, 2\lambda_2, \dots, 2\lambda_\ell$, and
all other (unmarked) faces of degree $4$. If $r$ is the number of unmarked
faces, such a map has $r+\ell$
faces, and by Euler's formula, the genus $g$ of this map satisfies:
$
v -2n + (r+\ell)=2-2g.
$
Moreover the number of edges is equal to the sum of the half face degrees so
$2n=2r+|\lambda|$, therefore we obtain the relation:
\begin{align}\label{eq:genusFromMonomial}
2g = n + 2 -v +\frac{|\lambda|}{2}-\ell,
\end{align}
which we shall use repeatedly. 
We now proceed with the proof of Lemma~\ref{lemma:only}.

\begin{proof}[Proof of \eqref{eq:22}]
As discussed above, $H_{2,2}$ is the weighted generating function of rooted bipartite maps with two
marked faces of degree $4$, so $\theta H_{2,2}$ is the weighted generating function of rooted
quadrangulations with two marked faces. Moreover,
by~\eqref{eq:genusFromMonomial}, the maps that contribute to the coefficient
$[z^{2n}w^{n+2-2g}]$ in $\theta H_{2,2}$ have genus $g$. Now, there are $n (n-1)$ ways of marking
two faces in a quadrangulation with $n$ faces, and the weight of such a map is
$\frac{1}{2n}$ since it  has $2n$ edges.
 Therefore:
$
[z^{2n} w^{n+2-2g}] \theta H_{2,2} =  \frac{1}{2n} n\cdot(n-1) Q_g^n(x). 
$
\end{proof}

\begin{proof}[Proof of \eqref{eq:11} and~\eqref{eq:1111}]
As discussed above, for $k\geq 1$, 
$\theta H_{1^{2k}}$ is the weighted generating function 
of bipartite maps carrying $2k$  marked (numbered) faces of degree
$2$, having all other faces of degree $4$.
Moreover, by~\eqref{eq:genusFromMonomial}, the genus of maps that contribute to
the coefficient $[z^{2n}w^{n+k-2g}]$ in this series is equal to $g+1-k$.
Therefore:
\begin{eqnarray}\label{eq:inter1}
[z^{2n} w^{n+k-2g}] \theta H_{1^{2k}} =  \frac{1}{2n} P_{g+1-k}^{2n,2k}(x)
\end{eqnarray}
where $P_h^{m,\ell}(x)$ denotes the generating polynomial (the variable $x$ marking white vertices) of rooted bipartite maps of genus $h$ with
$\ell$ numbered marked faces of degree $2$, all other faces of degree $4$, and
$m$ edges in total. Now, we claim that for all $h$ and all $m, \ell$ with $m+\ell$ even one has:
\begin{eqnarray}\label{eq:inter2}
P_h^{m,\ell}(x) = m (m-1) \dots (m-\ell+1) Q_h^{\frac{m-\ell}{2}}(x).
\end{eqnarray}
This is obvious for $\ell =0$ since a quadrangulation with $m$ edges has $m/2$
faces. For $\ell \geq 1$, consider a bipartite map with all faces
of degree $4$, except $\ell$ marked faces of degree $2$, and $m$ edges in total.
By contracting the first marked face into an edge, one obtains a map with one
less marked face, and a marked edge. This marked edge can be considered as the
root edge of that map (keeping the canonical bicolouration of vertices). Conversely, starting with a map having $\ell-1$ marked
faces, and $m-1$ edges, and expanding the root-edge into a face of degree $2$, 
there are $m$ ways of choosing a root corner in the resulting map in a way that
preserves the canonical bicolouration of vertices.
Since the contraction operation does not change the number of white vertices,  we deduce that $P_h^{m,\ell}(x)=m \cdot P_h^{m-1,\ell-1}(x)$ and~\eqref{eq:inter2} follows
by induction. \eqref{eq:11} and~\eqref{eq:1111} then follow
from~\eqref{eq:inter1} for $k=1$ and $k=2$, respectively.
\end{proof}

\begin{proof}[Proof of \eqref{eq:31}]
This case starts in the same way as the three others, but we will
have to use an additional tool (a simple Tutte equation) in order to express
everything in terms of quadrangulation numbers only.
First, $\theta H_{3,1}$ is
the weighted generating function of rooted bipartite maps
with one face of degree $6$, one face of degree $2$, and all other faces of degree
$4$.
Moreover, by~\eqref{eq:genusFromMonomial}, maps that contribute to the
coefficient of $[z^{2n}w^{n+2-2g}]$ in this series all have genus $g$.
We first get rid of the face of degree~$2$ by contracting it into an edge, and
declare this edge as the root of the new map, keeping the canonical
bicolouration. If
the original map has $2n$ edges, we obtain a map with $2n-1$ edges in total.
Conversely, if we start with a map with $2n-1$ edges and we expand the root
edge into a face of degree $2$, we have $2n$ ways of choosing a new root corner in
the newly created map, keeping the canonical bicolouration. Therefore 
if we let $X_g^n(x)$ be the generating polynomial (the variable $x$ marking white vertices) 
of rooted bipartite maps having a face
of degree $6$, all other faces of degree $4$, and $2n-1$ edges in total, we  have:
$$
[z^{2n} w^{n+2-2g}] \theta H_{3,1} = \frac{1}{2n} \cdot
2n X_g^n(x) = X_g^n(x),
$$
where the first factor is the weight coming from the definition of $H$.
Thus to prove~\eqref{eq:31} it is enough to establish the following equation:
\begin{eqnarray}\label{eq:tutteHexa}
Q_g^n(x)=
 \frac{3}{2n-1}X_g^{n}(x) + (1+x)Q_g^{n-1}(x).
\end{eqnarray}
The reader well acquainted with map enumeration may have recognized
in~\eqref{eq:tutteHexa} a (very simple
case of a) Tutte/loop equation. It is proved as follows.
Let $\mathfrak{q}$ be a rooted bipartite quadrangulation of genus $g$ with $n$
faces, and let $e$ be the root edge of $\mathfrak{q}$.
There are two cases: 1. the edge $e$ is bordered by two distinct faces, and 2.
the edge $e$ is bordered twice by the same face.

In case 1., removing the edge $e$ gives rise to a map
of genus $g$ with a marked face of degree $6$. By marking one of the $2n-1$
white corners of this map as the root, we obtain a rooted map counted by
$X_g^{n}(x)$, and since there are $3$ ways of placing a diagonal in a face of degree $6$
to create two quadrangles, the generating polynomial $N_1(x)$ corresponding to case 1.
satisfies $(2n-1)N_1(x) = 3 X_g^{n}(x)$. 

In case 2., the removal of the edge $e$
creates two faces (\emph{a priori}, either in the same or in two different connected
components) of degrees $k_1,k_2$ with $k_1+k_2+2=4$. Now since $\mathfrak{q}$ is bipartite, 
$k_1$ and $k_2$ are even which shows that one of the $k_i$ is zero and the
other is equal to $2$.
Therefore, in $\mathfrak{q}$, $e$ is a single edge hanging in a
face of degree~$2$. By removing $e$ and contracting
the degree $2$ face, we obtain a
quadrangulation with $n-1$ faces (and a marked edge that serves as a root,
keeping the canonical bicolouration). 
Conversely,  there are two ways to attach a hanging  edge
in a face of degree~$2$, which respectively keep the number of white vertices equal or increase it by one.
Therefore the generating polynomial corresponding to case 2. is $N_2(x)=(1+x)Q_g^{n-1}(x)$.

Writing that $Q_g^n(x)=N_1(x)+N_2(x)$, we obtain~\eqref{eq:tutteHexa} and complete the
proof.
\end{proof}

\begin{proof}[Proof of Theorem~\ref{thm:mainx}]
Just extract the coefficient of $[z^nw^{n+2-2g}]$ in Equation~\eqref{eq:thetakp} using
Lemma~\ref{lemma:only}, and group together the two terms containing $Q_g^n(x)$,
namely
$\frac{n-1}{2}Q_g^n(x) -\frac{2n-1}{3}Q_g^n(x) =  -\frac{n+1}{6}Q_g^n(x).$

\end{proof}

\section{Fixed genus generating functions}
\label{sec:gencor}

\subsection{The univariate generating functions $Q_g(t)$}
\label{sec:gencor1}
We start by studying the generating functions of maps of fixed genus by the number of edges.
In particular through the whole Section~\ref{sec:gencor1} we set $x=1$ and we use the notation $Q_g^n\equiv Q_g^n(1)$ as in Theorem~\ref{thm:main}.
We let $Q_g(t):=\sum_{n\geq 0} Q_g^n t^n$ be the generating function of rooted
maps of genus $g$ by the number of edges.  It was shown in~\cite{BC} that
$Q_g(t)$ is a rational function of $\rho:=\sqrt{1-12t}$.
In genus $0$, the result goes back to Tutte~\cite{Tutte:census} and one has the
explicit expression:
\begin{eqnarray}\label{eq:planar}
Q_0(t) = T - t T^3,
\end{eqnarray}
where $T=\frac{1-\sqrt{1-12t}}{6t}$ is the unique formal power series solution of the equation
\begin{eqnarray}\label{eq:T}
T=1+3tT^2.
\end{eqnarray}
In the following we will give a very simple recursive formula to compute the
series $Q_g(t)$ as a rational function of $T$, and we will study some of its
properties
\footnote{Note that being a rational function of $T$ or $\rho$ is equivalent,
but we prefer to work with $T$, since as a power series $T$ has a clear
combinatorial meaning. 
Indeed, $T$ is the generating function of labelled/blossomed trees, which are 
the fundamental building blocks
that underly the bijective decomposition of maps~\cite{Schaeffer,
ChassaingSchaeffer,CMS}. It is thus tempting to believe that those rationality
results have a combinatorial interpretation in terms of these
trees, even if it is still an open problem to find one. Indeed, so far the best rationality statement that is understood
combinatorially is that
the series of rooted bipartite quadrangulations of genus $g$ with a distinguished 
vertex is a rational function
in the variable $U$ such that $1=tT^2 (1+U+U^{-1})$, which is weaker than the
rationality in $T$. See~\cite{CMS} for this result.\label{footpage}}.

\begin{theorem}\label{thm:recgf}
For $g\geq 0$,  we have
$Q_g(t)=R_g(T)$ where $T$ is given by~\eqref{eq:T} and $R_g$ is a rational
function that can be computed iteratively
via:
\begin{align}\label{eq:diffrecT}
&\frac{d}{dT}\left(
\frac{(T-1)(T+2)}{3T} R_g (T)
\right)\\ &=
\frac{(T-1)^2}{18T^4}(2D+1)(2D+2)(2D+3)R_{g-1}(T)
+
 \frac{(T-1)^2}{3T^4} \sum_{i+j=g\atop i,j\geq 1}
\big((2 D+1) R_i(T)\big) 
\big((2 D+1) R_j(T)\big) \nonumber,
\end{align}
where $\displaystyle D=\frac{T(1-T)}{T-2}\frac{d}{dT}$.
\end{theorem}
\begin{proof}
First, one easily checks that Theorem~\ref{thm:main} is equivalent to the following differential
equation:
\begin{eqnarray}\label{eq:diff}
&&(D+1)Q_g = \\ &&4t (2D+1)Q_g +
\frac{1}{2}t^2(2D+1)(2D+2)(2D+3)Q_{g-1}+3t^2\sum_{i+j=g\atop i,j\geq 0}
\big((2D+1)Q_i\big)\big((2D+1)Q_j\big), \nonumber
\end{eqnarray}
where $D$ is the operator $D:=t\cdot\frac{d}{dt}$. Using~\eqref{eq:T} one
checks that $T'(t)=\frac{T(1-T)}{T-2}$, so that
$D=\left(\frac{dT(t)}{dt}\right) \frac{d}{dT}=\frac{T(1-T)}{T-2}\frac{d}{dT}$
and the definition of $D$ coincides with the one given in the statement of the
theorem.

Now, for $h\geq 0$ let $R_h$ be the unique formal power series such that
$Q_h(t)=R_h(T)$.
Grouping all the genus $g$ generating functions on the left hand side, 
we can put~\eqref{eq:diff} in the form:
\begin{eqnarray}\label{eq:diff1}
A R_g(t) + B \frac{d}{dT} R_g(t) = R.H.S.
\end{eqnarray}
where $A = 1-4t-6t^2 (2D+1)Q_0$, $B=t(1-8t-12t^2(2D+1)Q_0))$, and the R.H.S. is
the same as in~\eqref{eq:diffrecT}.
Using the explicit expression~\eqref{eq:planar} of $Q_0$ in terms of $T$, 
we can then rewrite the L.H.S. of~\eqref{eq:diff1} as
\begin{eqnarray*}
\frac{(T-1)(T+2)}{3T} R'_g(T) + \frac{T^2+2}{3T^2} R_g(T) = 
\frac{d}{dT} \left(  \frac{(T-1)(T+2)}{3T}R_g(T)
\right),
\end{eqnarray*}
and we obtain~\eqref{eq:diffrecT}. Note that we have not proved that $R_g(T)$
is a rational function: we admit this fact from~\cite{BC}.
\end{proof}

Observe that we have $R_g(1)=Q_g(0)<\infty$ so the quantity $\frac{(T-1)(T+2)}{3T}
R_g (T)$ vanishes at $T=1$, and we have:
$$\frac{(T-1)(T+2)}{3T} R_g (T) = \int_1^T R.H.S.,$$
with the R.H.S. given by~\eqref{eq:diffrecT},
which shows that~\eqref{eq:diffrecT} indeed enables one to compute the $R_g$'s
recursively.
Note that it is not obvious {\it a priori} that no logarithm appears during this
integration, although this is true since it is known that $R_g$ is
rational\footnote{We unfortunately haven't been able to reprove this fact from
our approach}~\cite{BC}.
Moreover, since all generating functions considered are finite at $T=1$ (which
corresponds to the point $t=0$) we obtain via an easy induction that $R_g$ has only poles at $T=2$
or $T=-2$. More precisely, by an easy induction, we obtain a bound on the
degrees of the poles:
\begin{corollary}
For $g\geq 1$ we have $Q_g(t)=R_g(T)$ where $R_g$ can be written as:
\begin{align}\label{eq:ansatz}
R_g =
c_0^{(g)}+ 
\sum_{i=1}^{5g-3} \frac{\alpha_i^{(g)}}{(2-T)^i}
+\sum_{i=1}^{3g-2} \frac{\beta_i^{(g)}}{(T+2)^i},
\end{align}
for rational numbers $c_0^{(g)}$ and $\alpha_i^{(g)},\beta_i^{(g)}$.
\end{corollary}
Note that by plugging the ansatz~\eqref{eq:ansatz} into the
recursion~\eqref{eq:diffrecT}, we obtain a very efficient way of computing the
$R_g$ 's inductively. 

We conclude this section with (known) considerations on asymptotics.
From~\eqref{eq:ansatz}, 
it is easy to see that the dominant singularity of $Q_g(t)$ is unique, and is reached at
$t=\frac{1}{12}$, \emph{i.e.} when $T=2$. In particular the dominant term
in~\eqref{eq:ansatz} is $\frac{\alpha_{5g-3}^{(g)}}{(2-T)^{5g-3}}$.
Using the fact that $2-T=2\sqrt{1-12t}+O(1-12t)$ when $t$ tends
to~$\frac{1}{12}$, and using a standard transfer theorem for algebraic
functions~\cite{FO}, we obtain
that for fixed $g$, $n$ tending to infinity:
\begin{align}\label{eq:asympt}
Q_g^n \sim t_g n^{\frac{5(g-1)}{2}}12^n,
\end{align}
with $t_g = \frac{1}{2^{5g-3}\Gamma(\frac{5g-3}{2})}\alpha_{5g-3}^{(g)}$.
Moreover, by extracting the leading order coefficient in~\eqref{eq:diffrecT}
when $T\sim 2$, 
we see with a short computation that the sequence 
$\tau_g=(5g-3)\alpha_{5g-3}^{(g)}=2^{5g-2}\Gamma(\frac{5g-1}{2})t_g$ satisfies 
the following Painlevé-I type recursion
\begin{align}\label{eq:tgrec}
\tau_g = \frac{1}{3} (5g-4)(5g-6)\tau_{g-1}+\frac{1}{2} \sum_{h=1}^{g-1}
\tau_{h} \tau_{g-h},
\end{align}
which enables one to compute the $t_g$'s easily by induction starting from
$t_1=\frac{1}{24}$ ({\it i.e.} $\tau_1=\frac{1}{3}$).
These results are well known (for~\eqref{eq:asympt} see \cite{BC0}, or \cite{CMS, Chapuy:PTRF} for bijective interpretations;
for~\eqref{eq:tgrec} see \cite[p.201]{LandoZvonkine} for historical references, or~\cite{BGR} for a proof along the same lines as ours but starting from the Goulden and Jackson recurrence~\cite{GJ2008}).
So far, as far as we know, all the known proofs of~\eqref{eq:tgrec} rely on the use of integrable
hierarchies.

\subsection{Genus generating function at fixed $n$.} We now indicate a straightforward consequence of Theorem~\ref{thm:mainx} in terms of ''genus generating functions'', {\it i.e.} generating functions of maps where the number of edges is fixed and genus varies. Such generating functions have been considered in the combinatorial literature before (with a slightly different scaling), especially in the case of one-face maps, where they admit elegant combinatorial interpretations (see~\cite{OB}).
They also appear naturally in formal expansions of matrix integrals (see, e.g., \cite{LandoZvonkine} or \cite{HZ}).
 Let $s$ be a variable, and for $n\geq1$ let $H_n(x,s)$ be the generating polynomial of maps with $n$ edges by the number of faces (variable $x$), and the genus (variable $s$), {\it i.e.}:
$$
H_n(x,s)=\sum_{g \geq 0} Q_g^n(x) s^g,
$$
in the notation of Theorem~\ref{thm:mainx}. Then Theorem~\ref{thm:mainx} has the following equivalent formulation:
\begin{corollary}\label{cor;genusgen}
The genus generating function $H_n\equiv H_n(x,s)$ is solution of the following recurrence equation:
\begin{multline*}
\!\!
\!\!
\!\!
\frac{n+1}{6} H_n = \frac{(1+x)(2n-1)}{3} 
  H_{n-1} + \frac{s ( 2n-3 )  ( 2n-2 )  ( 2n-1 ) }{12}  H_{n-2} \\ 
+
\frac{1}{2}\sum _{k+\ell=n\atop k,\ell\geq 1}
 ( 2k-1 )  ( 2\ell-1 )   H_{k-1} H_{ \ell-1},\hspace{1cm} 
\end{multline*}
for $n\geq 1$, with the initial condition $H_0(x,s)=x$.
\end{corollary}

\subsection{The bivariate generating functions $M_g(x;y)$}
\label{sec:gencor2}
We let  $M_g^{i,j}$ be defined as in Theorem~\ref{thm:three2} and we consider the bivariate generating function of rooted maps of genus $g$ by vertices (variable~$x$) and faces (variable $y$):
$$
M_g(x,y) = \sum_{i,j\geq 1} M_g^{i,j}x^iy^j.
$$ 
Bender, Canfield, and Richmond~\cite{BCRvf} showed that $M_g(x,y)$ is a rational function in the two parameters $p$ and $q$ such that:
\begin{eqnarray}\label{eq:birat}
x=p(1-p-2q), ~~ y=q(1-2p-q).
\end{eqnarray}
Arqu\`es and Giorgetti~\cite{AR} later refined this result and showed that 
\begin{eqnarray} \label{eq:formbirat}
M_g(x,y) = \frac{pq(1-p-q)P_g(p,q)}{\big((1-2p-2q)^2-4pq\big)^{5g-3}}
\end{eqnarray} where $P_g$ is a polynomial of total degree at most $6g-6$. However, similarly as in the case of univariate functions discussed in the previous section, the recursions given in~\cite{BCRvf, AR} to compute the polynomials $P_g$ involve additional variables and are complicated to use except for the very first values of $g$.

It would be natural to try to compute these polynomials by reformulating Theorem~\ref{thm:three2} as a recursive partial differential equation for the series $M_g(x,y)$ in the variables $(x,y)$ or $(u,v)$, in the same way as we did for the univariate series in Section~\ref{sec:gencor1}. However due to the bivariate nature of the problem, this approach does not seem to lead easily to an efficient way of computing the polynomials $P_g(u,v)$. 

Instead, we prefer to remark that since $P_g(u,v)$ has total degree at most $(6g-6)$ one can use the method of undetermined coefficients. In order to determine $P_g(u,v)$ we need to determine ${6g-5 \choose 2}$ coefficients, which can be done by computing the same number of terms of the sequence $M_g^{i,j}$. More precisely it is easy to see that computing $M_g^{i,j}$ for all $i,j$ such that $2\leq i+j\leq 6g-4$ gives enough data to determine the polynomial $P_g$, whose coefficients can then be obtained by solving a linear system. Since Theorem~\ref{thm:three2} gives a very efficient way of computing the numbers $M_g^{i,j}$, this technique seems more efficient (and much simpler) than trying to solve recursively the functional equations of the papers~\cite{BCRvf, AR}. We have implemented it on Maple and checked that we recovered the expressions of~\cite{AR} for $g=1,2,3$, and computed the next terms with no difficulty (for example computing $P_{10}(p,q)$ took less than a minute on a standard computer).

\section{Discussion and comparison with other approaches}
\label{sec:comments}

In this paper we have obtained simple recurrence formulas to compute the
numbers of rooted maps of genus, edges, and optionally vertices. It
gives rise in particular to efficient inductive formulas to compute the fixed genus
generating functions.
Let us now compare with other existing approaches to enumerate maps on
surfaces.

\medskip

\noindent{\bf Tutte/loop equations}. The most direct way to count maps on
surfaces is to perform a root edge decomposition, whose counting counterpart is 
known as \emph{Tutte
equation} (or \emph{loop equation} in the context of mathematical physics).
This approach enabled Bender and Canfield~\cite{BC} to prove the rationality of the
generating function of maps in terms of the parameter $\rho$ as discussed
in Section~\ref{sec:gencor1}. It was generalized to other classes of maps  {\it
via} variants of the kernel method (see, e.g.,
\cite{Gao}), and to the bivariate case in~\cite{BCRvf, AR} as discussed in Section~\ref{sec:gencor2}.
This approach has been considerably improved by the Eynard school
(see e.g. \cite{Eynard:book}) who  developed powerful machinery to solve
recursively these equations for many families of maps.
 
However, because they are based on Tutte equations, both the methods of~\cite{BC, Gao, BCRvf, AR} 
and~\cite{Eynard:book} require
working with generating functions of maps carrying an arbitrarily large number of
additional boundaries. To illustrate, in the special case of quadrangulations, 
the ``topological recursions'' given by
these papers enable one to compute inductively the generating functions 
$Q_g^{(p)}(t) \equiv Q_g(t;x_1,x_2,\dots,x_p)$ of rooted quadrangulations of genus $g$ carrying $p$
additional faces of arbitrary degree, marked by the additional variables
$x_1,x_2,\dots,x_p$. In order to be able to compute $Q_g^{(p)}(t)$ these
recursions take as an input the planar generating function $Q_0^{(g+k)}(t)$, so
one cannot avoid working with these extra variables (linearly many of them with
respect to the genus), even to compute the pure quadrangulation series
$Q_g^{(0)}$. 

Compared to this, the recurrence relations obtained in this paper
(Theorems~\ref{thm:mainx},~\ref{thm:main}, and~\ref{thm:recgf}) are much more efficient, as they
do not need to introduce any extra parameters.
In particular we can compute all univariate generating functions easily, even for large $g$.
However, of course, what we do here is a very special case: we consider only
bipartite quadrangulations keeping track of two or three parameters,
 whereas the aforementioned approaches enable one to
count maps with arbitrary degree distribution!

\medskip
\noindent{\bf Integrable hierarchies.} It has been known for some time in the
context of mathematical physics that multivariate generating functions of maps
are solution of integrable hierarchies of partial differential equations such
as the KP or the Toda hierarchy, see e.g.~\cite{Okounkov,Orlov, LandoZvonkine, GJ2008}.
However these hierarchies do not characterize their solutions (as shown by the
fact that many combinatorial models give different solutions), and one needs to
add extra information to compute the generating functions inductively.
Let us mention three interesting situations in which this is possible.
The first one is Okounkov's work on Hurwitz numbers~\cite{Okounkov}, where the
integrable hierarchy is the 2-Toda hierarchy, and the ``extra information''
takes the form of the computation of a commutator 
of operators 
in the infinite wedge
space~\cite[section 2.7]{Okounkov}.

 The second one is Goulden and Jackson's
recurrence for triangulations~\cite[Theorem 5.4]{GJ2008}, which looks very similar
to our main result. The starting equation is the same as ours
(Equation~\eqref{eq:kp}), but for the generating function of ordinary (non
bipartite) maps. In order to derive a closed equation from it, the authors
of~\cite{GJ2008} do complicated manipulations of generating functions, but what
they do could equivalently be done {\it via}
local manipulations similar to the ones we used in the proofs
of~\eqref{eq:11}, \eqref{eq:22}, \eqref{eq:1111}. We leave as an exercise to the
reader the task of reproving \cite[Theorem 5.4]{GJ2008} along these lines (and
with almost no computation).

The last one is the present paper, where in
addition to such local manipulations, we use an additional, very degenerate,
Tutte equation (Equation~\eqref{eq:tutteHexa}).
It seems difficult to find other cases than triangulations and
bipartite quadrangulations where the same techniques would apply, even by
allowing the use of more complicated Tutte equations. In our current
understanding, this situation is a bit mysterious to us.

To conclude on this aspect, let us observe that the equations obtained from
integrable hierarchies rely on the deep algebraic structure of the multivariate
generating series of combinatorial maps (and on their link with Schur functions). This structure
provides them with many symmetries that are not apparent in the combinatorial
world, and  we are far from understanding combinatorially the meaning of
these equations.
In particular, to our knowledge, the approaches based on integrable hierarchies are the only
ones that enable one to prove statements such as~\eqref{eq:tgrec}.
\medskip

\noindent{\bf Bijective methods.}
In the planar case ($g=0$) the combinatorial structure of maps is now well
understood thanks to bijections that relate maps to some kinds of decorated
trees. The topic was initiated by Schaeffer~\cite{Schaeffer, ChassaingSchaeffer}
and has been developed by many others. 
For these approaches, the simplest case turns out to be the one of
bipartite quadrangulations. In this case, the trees underlying the bijective
decompositions have a generating function given by~\eqref{eq:T}.

The bijective combinatorics of maps on other orientable surfaces is a more recent topic.
Using bijections similar to the ones in the planar case, one can prove
bijectively
rationality results for the fixed-genus generating function of
quadrangulations~\cite{CMS} or more generally fixed degree bipartite maps or
constellations~\cite{Chapuy:constellations}. However, with these techniques, one 
obtains rationality in terms of some auxiliary generating functions whose degree of
algebraicity is in general too high compared to the known non-bijective result.
See the footnote page~\pageref{footpage} for an example of this phenomenon in
the case of quadrangulations.
Moreover, although the asymptotic form~\eqref{eq:asympt} is well explained
by these methods~\cite{CMS, Chapuy:constellations, Chapuy:PTRF}, they
do not provide any information on the numbers $t_g$, and do not explain the
relation~\eqref{eq:tgrec}.
In a different direction, much progress has been made in recent years on the combinatorial understanding of one-face maps, and we now have convincing bijective proofs, for example, of the Harer-Zagier recurrence formula (Corollary~\ref{cor:HZ}). See e.g.~\cite{CFF} and references therein.

To conclude,  we are still far from being able to prove exact counting statements
such as Theorem~\ref{thm:main} or Theorem~\ref{thm:three2} combinatorially. However, the fact that Theorem~\ref{thm:three2} contains the Harer-Zagier formula, which is well understood combinatorially, as a special case, opens an interesting track of research.
Moreover, the history of bijective methods for maps tells us two things. First,
that when a bijective approach exists to some map counting problem, the case of bipartite quadrangulations
is always the easiest one to start with. Second, that before trying to find
bijections, it is important to know \emph{what to prove} bijectively.
Therefore we hope that, in years to come, the results of this paper will play a role  guiding new
developments of the bijective approaches to maps on surfaces.

\bibliographystyle{plain}
\bibliography{document}

\end{document}